\documentclass[a4paper]{amsart}
\usepackage{bbding}
\usepackage{mathrsfs}
\usepackage{cite}
\usepackage{enumerate}
\usepackage{url}

\newtheorem{theorem}{Theorem}[section]
\newtheorem{lemma}[theorem]{Lemma}
\newtheorem{prop}[theorem]{Proposition}
\newtheorem{corollary}[theorem]{Corollary}

\newtheorem{example}[theorem]{Example}

\numberwithin{equation}{section}

\newcommand{\R}{{\mathbb{R}}}

\renewcommand{\d}{\partial}

\newcommand{\lin}{{\operatorname{lin }}}
\newcommand{\aff}{{\operatorname{aff }}}
\newcommand{\inte}{{\operatorname{int }}}

\newcommand{\cl}{{\operatorname{cl}}}
\newcommand{\ve}{{\varepsilon}}
\newcommand{\pa}{{\partial}}
\newcommand{\relbd}{{\operatorname{relbd}}}
\newcommand{\Om}{{\Omega}}
\newcommand{\om}{{\omega}}
\newcommand{\veps}{{\varepsilon}}

\begin{document}

\title[Smoothness in the  $L_p$ Minkowski problem]{Smoothness in the  $L_p$ Minkowski problem \\ for $p<1$}
\author[Gabriele Bianchi, K\'aroly J. B\"or\"oczky, Andrea Colesanti]
{Gabriele Bianchi, K\'aroly J. B\"or\"oczky, and Andrea Colesanti}
\address{Dipartimento di Matematica e Informatica ``U. Dini", Universit\`a di Firenze, Viale Morgagni 67/A, Firenze, Italy I-50134} \email{gabriele.bianchi@unifi.it}
\address{Alfr\'ed R\'enyi Institute of Mathematics, Hungarian Academy
  of Sciences, Reltanoda u. 13-15, H-1053 Budapest, Hungary, and
Department of Mathematics, Central European University, Nador u 9, H-1051, Budapest, Hungary} \email{boroczky.karoly.j@renyi.mta.hu}
\address{Dipartimento di Matematica e Informatica ``U. Dini", Universit\`a di Firenze, Viale Morgagni 67/A, Firenze, Italy I-50134} \email{andrea.colesanti@unifi.it}
\thanks{First and third authors are supported in part by the Gruppo Nazionale per l'Analisi Matematica, la Probabilit\`a e le loro Applicazioni (GNAMPA) of the Istituto Nazionale di Alta Matematica (INdAM). Second author is supported in part by
NKFIH grants 116451 and 109789.}
%\subjclass[2010]{Primary: 52A20, 52A39; secondary: 28B20, 52A38, 52A40} 

\begin{abstract} 
We discuss the smoothness and strict convexity of the solution of the $L_p$-Minkowski problem when $p<1$
and the given measure has a positive density function.
\end{abstract}

\maketitle

\section{Introduction}

Given  $K$ in the class $\mathcal {K}_0^n$ of compact convex sets in $\R^n$ that have non-empty interior and  contain the origin $o$, we write $h_K$ and $S_K$ to denote its support function and its surface area measure, respectively, 
and for $p\in\R$, $S_{K,p}$ to denote its  $L_p$-area measure, where $d S_{K,p}=h_K^{1-p}dS_K$.
The $L_p$-area measure defined by Lutwak \cite{Lut} is a central notion in convexity, see say 
Barthe, Gu\'{e}don, Mendelson and Naor \cite{BG}, B\"{o}r\"{o}czky, Lutwak, Yang and Zhang \cite{BLYZ2}, 
Campi and Gronchi \cite{CG},  Chou \cite{KSC},
Cianchi, Lutwak, Yang and Zhang \cite{CLYZ}, Gage and Hamilton \cite{GH}, Haberl and Parapatits \cite{HP},  Haberl and Schuster \cite{HS1,HS2},  Haberl, Schuster and Xiao \cite{HSX}, He, Leng and Li \cite{HLK}, 
Henk and Linke \cite{HENK}, Ludwig \cite{LU2}, Lutwak, Yang and Zhang \cite{LYZ1,LYZ4}, Naor \cite{NAO}, Naor and Romik \cite{NR}, 
Paouris \cite{PAO}, Paouris and Werner \cite{PW} and Stancu \cite{ST3}.

The $L_p$ Minkowski problem asks for the existence of a convex body $K\in\mathcal{K}_{0}^n$ whose $L_p$ area measure is a given finite Borel measure $\nu$ on $S^{n-1}$. 
When $p=1$ this is the classical Minkowski problem solved by Minkowski \cite{MIN} for polytopes, and by Alexandrov \cite{Ale38} and Fenchel and Jessen \cite{FeJ38} in general. The smoothness of the solution was clarified in a series of papers by Nirenberg \cite{NIR}, Cheng and Yau \cite{CY}, Pogorelov \cite{POG} and Caffarelli \cite{Caf90a,Caf90b}.
For $p>1$ and $p\neq n$, the $L_p$ Minkowski problem has a unique solution
according to Chou and Wang~\cite{CW}, Guan and Lin \cite{GL} and
Hug, Lutwak, Yang and Zhang \cite{HLYZ2}. 
The smoothness of the solution is discussed in Chou and Wang \cite{CW}, Huang and Lu \cite{HL1} and Lutwak and Oliker \cite{LO}.
In addition, the  case $p<1$ has been intensively investigated by B\"{o}r\"{o}czky, Lutwak, Yang and Zhang \cite{BLYZ}, B\"or\"oczky and Hai T. Trinh \cite{BoT},  Chen \cite{WC},  Chen, Li and Zhu \cite{CSL0,CSL01}, Ivaki \cite{Iva13}, Jiang \cite{JI}, Lu and Wang \cite{LWA}, Lutwak, Yang and Zhang \cite{LYZ5}, Stancu \cite{ST1,ST2} and Zhu \cite{Z2,Z3,Zhu15,Zhu16}.

The solution of the $L_p$-Minkowski problem may not be unique for $p<1$
according to Chen, Li and Zhu \cite{CSL01} if $0<p<1$, according to Stancu \cite{ST2} if $p=0$, and according to Chou and Wang \cite{CW} if $p<0$ small.
 
In this paper we are interested in this problem when $p<1$ and $\nu$ is a measure with density with respect to the Hausdorff measure  $\mathcal{H}^{n-1}$ on $S^{n-1}$, i.e. in the problem
\begin{equation}\label{Kdensityfunction}
dS_{K,p}=f\,d{\mathcal H}^{n-1}\quad\text{ on $S^{n-1}$,}
\end{equation}
where $f$ is a non-negative Borel function in $S^{n-1}$.

According to Chou and Wang~\cite{CW}, if $-n<p<1$ and the Borel function $f$ is bounded from above and below by positive constants, then~\eqref{Kdensityfunction} has a solution. More general existence results are provided by the recent works Chen, Li and Zhu \cite{CSL0} if $p=0$, Chen, Li and Zhu \cite{CSL01} if $0< p<1$, and  Bianchi, B\"or\"oczky and Colesanti~\cite{BiBoCoY} if $-n<p<0$. In particular, it is known that \eqref{Kdensityfunction} has a solution 
 if $0\leq p<1$ and
$f$ is any non-negative function in  $L_1(S^{n-1})$ with $\int_{S^{n-1}}f\,d{\mathcal H}^{n-1}>0$,
and if $-n <p<0$ and
$f$ is any non-negative function in  $L_{\frac{n}{n+p}}(S^{n-1})$ with $\int_{S^{n-1}}f\,d{\mathcal H}^{n-1}>0$.

We observe that $h$ is a  non-negative positively $1$-homogeneous convex function in $\R^n$ which solves the Monge-Amp\`ere equation 
\begin{equation}\label{MongeAmpere_Sn}
h^{1-p}\det(\nabla^2h+h I)=f \quad \text{ on $S^{n-1}$}
\end{equation} 
in the sense of measure if and only if 
$h$ is the support function of a convex body $K\in{\mathcal K}_{0}^n$ which is the solution of~\eqref{Kdensityfunction} (see Section~\ref{secPreliminaries}). 
Here $h$ is the unknown  non-negative (support) function on $S^{n-1}$ to be found, $\nabla^2 h$ denote the (covariant) Hessian matrix of $h$ with respect to an orthonormal frame on $S^{n-1}$, and $I$ is the identity matrix.
The function $h$ may vanish somewhere even in the case when $f$ is positive and continuous, and when this happens and $p<1$ the equation~\eqref{MongeAmpere_Sn} is singular at the zero set of $h$.
Naturally, if $h$ is $C^2$, then (\ref{MongeAmpere_Sn}) is a proper Monge-Amp\`ere equation.

In this paper we study the smoothness and strict convexity of a solution $K\in{\mathcal K}_{0}^n$ of~\eqref{Kdensityfunction}  assuming $\tau_2>f>\tau_1$ for some constants $\tau_2>\tau_1>0$.
Concerning these aspects for $p<1$, we summarize the known results in Theorem~\ref{th_regularity}, and the new results in Theorem~\ref{th_regularity-new}. 
 
We say that $x\in\partial K$ is a $C^1$-smooth point if there is a unique tangent hyperplane to $K$ at $x$, and observe that
$\partial K$ is $C^1$ if and only if each $x\in\partial K$ is $C^1$-smooth  (see Section~\ref{secPreliminaries} for all definitions). In addition, we note that $h_K$ is $C^1$ on $S^{n-1}$ if and only if $K$ is strictly convex, and $h_K$ is strictly convex on any hyperplane avoiding the origin if and only if $\partial K$ is $C^1$.
For $z\in\partial K$, the exterior normal cone at $z$ is denoted by $N(K,z)$, and for $z\in{\rm int}\,K$, we set $N(K,z)=\{o\}$. Theorem~\ref{th_regularity} (i) and (ii) are essentially due to
Caffarelli \cite{Caf90a} (see Theorem~\ref{Caffarelli-smooth}), and Theorem~\ref{th_regularity} (iii) is due to Chou and Wang \cite{CW}. If the function $f$ in \eqref{Kdensityfunction} is $C^\alpha$ for $\alpha>0$, then
Caffarelli \cite{Caf90b} proves (iv). 
 
\begin{theorem}[Caffarelli, Chou, Wang]
\label{th_regularity}
If $K\in{\mathcal K}_{0}^n$ is a solution of   \eqref{Kdensityfunction} for $n\geq 2$ and $p<1$, and $f$ is bounded from above and below by positive constants, then the following assertions hold:
\begin{enumerate}[(i)] 
\item\label{th_regularity_a_} The set $X_0$ of the points $x\in \partial K$ with $N(K,x)\subset N(K,o)$ is closed, 
each point of $X=\partial K\backslash X_0$ is $C^1$-smooth and $X$ contains no segment.
\item \label{th_regularity_b_} If  $o\in\partial K$ is a $C^1$-smooth point, then $\partial K$ is $C^1$.
\item \label{th_regularity_c_} If $p\leq 2-n$, then $o\in{\rm int}\,K$, and hence $K$ is strictly convex and $\partial K$ is $C^1$.
\item \label{th_regularity_x_} If $o\in{\rm int}\,K$ and the function $f$ in \eqref{Kdensityfunction} is positive and $C^\alpha$, for some $\alpha>0$, then $\partial K$ is $C^{2,\alpha}$.
\end{enumerate} 
\end{theorem}

Concerning strict convexity, Assertion~(iii) here  is optimal  because  Example~\ref{non-strictly-convex} shows that if 
$2-n<p<1$, then  it is possible that $o$ belongs to the relative interior of an $(n-1)$-dimensional face of 
a solution $K$ of \eqref{Kdensityfunction} where $f$ is a positive continuous function.
Therefore, the only question left open is the $C^1$ smoothness of the boundary of the solution if $2-n<p<1$.

We note that if $p<1$ and $K$ is a solution of \eqref{MongeAmpere_Sn} with $f$ positive and $o\in\partial K$, then 
\begin{equation}
\label{NKosmall}
{\rm dim}\,N(K,o)\leq n-1.
\end{equation}
Therefore, Theorem~\ref{th_regularity}  (ii) yields that $\partial K$ is $C^1$ for the solution $K$ if $n=2$. In general, we have the following partial results.

\begin{theorem}
\label{th_regularity-new}
If $K\in{\mathcal K}_{0}^n$ is a solution of   \eqref{Kdensityfunction} for $n\geq 2$ and $p<1$, and $f$ is bounded from above and below by positive constants, then the following assertions hold:
\begin{enumerate}[(i)] 
\item\label{th_regularity_e_} If $n=2$, $n=3$ or $n>3$ and $p<4-n$, then $\partial K$ is $C^1$.
\item\label{th_regularity_f_} If $\mathcal{H}^{n-1}(X_0)=0$ for the $X_0$ in Theorem~\ref{th_regularity} (i), then $\partial K$ is $C^1$.
\end{enumerate} 
\end{theorem}

Our results  differ in some cases from the ones in Chou and Wang \cite{CW}, possibly because \cite{CW} considers the equation 
\begin{equation}\label{MongeAmpere_CW}
\det(\nabla^2h+h I)=f h^{p-1}\quad \text{ on $S^{n-1}$}
\end{equation}
instead of \eqref{MongeAmpere_Sn}. In the context of non-negative convex functions, being a solution of this last equation is a priori more restrictive than being a solution of \eqref{MongeAmpere_Sn}, even if obviously the  two notions coincide when $h$ is positive (see Section~\ref{secPreliminaries} for more on this point). 
Chou and Wang \cite{CW} proves,  under our same assumptions on $f$, the strict convexity of  the solution $h$ of \eqref{MongeAmpere_CW} on hyperplanes avoiding the origin, and uses this to prove that $\partial K$ is $C^1$
for the convex body $K$.
We note that if $K\in{\mathcal K}_{0}^n$ is a solution of \eqref{MongeAmpere_CW} for $p<1$ and $f$ is bounded from below and above by positive constants, then
combining Theorem~\ref{th_regularity-new}~(ii) with the simple observation \eqref{sol_alex_c} in Section~\ref{secPreliminaries}
shows that $\partial K$ is $C^1$, as it was verified by Chou and Wang~\cite{CW}. 
In our opinion \eqref{MongeAmpere_Sn} is the right equation to consider and using it we obtain weaker results. 

To give an example of how the two equations differ, the support function $h$ of the body $K$ in Example~\ref{non-strictly-convex} (where $o$ belongs to the relative interior of an $(n-1)$-dimensional face)  is a solution of \eqref{MongeAmpere_Sn} but not a solution of \eqref{MongeAmpere_CW}.

According to Chou and Wang~\cite{CW} (see also Lemma~\ref{MongeAmpereRn-lemma} below), the Monge-Amp\`ere equation (\ref{MongeAmpere_Sn}) can be transferred to a 
 Monge-Amp\`ere equation
\begin{equation}\label{MongeAmpereRn0}
v^{1-p}\det( D^2 v )=g
\end{equation}
for a convex function $v$ on $\R^{n-1}$ where $g$ is a given non-negative function and $D^2$ stands for the Hessian in $\R^{n-1}$.

The proofs of Claims~\eqref{th_regularity_a_} and~\eqref{th_regularity_b_} in Theorem~\ref{th_regularity} use as an essential tool a result proved by Caffarelli in \cite{Caf90a} regarding smoothness and strict convexity of convex solutions of  certain Monge-Amp\`ere equation of type \eqref{MongeAmpereRn0} (see Theorem~\ref{Caffarelli-smooth}). Proving that $\pa K$ is $C^1$ is equivalent to prove that $h_K$ is strictly convex, and \cite{Caf90a} is the key to prove this property in $\{y\in S^{n-1} : h_K(y)>0\}$. 

The proof of Claim~\eqref{th_regularity_e_} in Theorem~\ref{th_regularity-new} is based on the following result for the singular inequality $v^{1-p}\det D^2 v\geq g$.

\begin{prop}\label{vanishingonsegment} 
Let $\Omega\subset\R^n$ be an open convex set, and let $v$ be a non-negative convex function 
in $\Omega$ with $S=\{x\in\Omega:\,v(x)=0\}$. If for $p<1$ and $\tau>0$, $v$ is the solution of 
\begin{equation}\label{MongeAmpereOmega}
v^{1-p}\det D^2 v \geq \tau \quad \text{ in $\Omega\setminus S$}
\end{equation}
in the sense of measure, and $S$ is $r$-dimensional, for $r\geq 1$, then $p\geq-n+1+2r$.
\end{prop}

We mention that in Caffarelli \cite{Caf93} a corresponding result for $p=1$ is established.

The underlying idea behind the proof of this result is the following: On the one hand, the graph of $v$ near $S$ is close to being ruled. Hence, the total variation of the derivative is ``small". On the other hand, the total variation of the derivative is ``large" because of the Monge-Amp\`ere inequality (\ref{MongeAmpereOmega}).

The inequality $p\geq-n+1+2r$ in this result is close to being optimal, at least when $r=1$. Indeed, Example~\ref{example_vanishingonsegment} shows that, for any $p>-n+3$, there exists   a non-negative convex solution of \eqref{MongeAmpereOmega} in $\Omega$ which vanishes on the intersection of $\Omega$ with a line.
For the version $p=1$ of Proposition~\ref{vanishingonsegment}, Caffarelli \cite{Caf93}
proves that ${\rm dim}\,S<n/2$ and that this inequality is optimal.

Proposition~\ref{vanishingonsegment}  yields actually somewhat more than Claim~\eqref{th_regularity_e_} in Theorem~\ref{th_regularity-new}; namely, if $r\geq 2$ is an integer, $p<\min\{1,2r-n\}$ and 
$K\in{\mathcal K}_{0}^n$ is a solution of   \eqref{Kdensityfunction} with $o\in\partial K$, then 
${\rm dim}\,N(K,o)<r$. As a consequence, we have the following technical statements about $K$, where we also use Theorem~\ref{th_regularity-new}~\eqref{th_regularity_f_} for Claim~(ii).

\begin{corollary}\label{cor_regularity_new}
If $p<1$ and $K\in{\mathcal K}_{0}^n$, $n\geq 4$, is a solution of   \eqref{Kdensityfunction} with $o\in\partial K$, then 
\begin{enumerate}[(i)] 
\item\label{th_regularity_g_} ${\rm dim}\,N(K,o)<\frac{n+1}2$;
\item\label{th_regularity_h_} if in addition $n=4,5$ and $\partial K$ is not $C^1$, then 
${\rm dim}\,N(K,o) =2$ and ${\rm dim}\,F(K,u)=n-1$ for some $u\in N(K,o)$.
\end{enumerate} 
\end{corollary}

In Section~\ref{secPreliminaries} we review the notation used in this paper.
Section~\ref{sec-Monge-Ampere} contains results and examples regarding Monge-Amp\`ere equations in $\R^n$, namely Proposition~\ref{vanishingonsegment}, Example~\ref{example_vanishingonsegment} and Proposition~\ref{ustrictconvex}. This last result is the key to prove Theorem~\ref{th_regularity-new}~\eqref{th_regularity_h_}.
In Section~\ref{sec-th_regularity} we show, for the sake of completeness, how to prove Theorem~\ref{th_regularity} using ideas due to Caffarelli \cite{Caf90a,Caf90b}
and Chou and Wang~\cite{CW}.
Theorem~\ref{th_regularity-new} and Corollary~\ref{cor_regularity_new}   are proved in Section~\ref{sec-th_regularity-new}.

\section{Notation and preliminaries}
\label{secPreliminaries}

As usual, $S^{n-1}$ denotes the unit sphere and  $o$ the
origin in the Euclidean $n$-space $\R^n$. The symbol $B^n$ denotes the unit ball in $\R^n$ centred at $o$ and $\omega_n$ denotes its volume. If $x,y\in\R^n$, then $\left<x,y\right>$ is the scalar product of $x$ and $y$, while $\|x\|$ is the euclidean norm of $x$. By $[x,y]$ we denote the segment with endpoint $x$ and $y$.
%Given $x$ and $y\in S^{n-1}$ the symbol $\angle(x,y)$ denotes the length of the geodesic arc in $S^{n-1}$ with endpoint $x$ and $y$. If $\alpha\in(0,\pi)$ let $\Omega(x,\alpha)=\{y\in S^{n-1} : \angle(x,y)< \alpha\}$.

We write $\mathcal{H}^k$ for $k$-dimensional Hausdorff measure in $\R^n$.

We denote by $\partial E$,  $\inte E$,  $\cl E$, and $1_E$ the {\it boundary}, 
{\it interior}, \emph{closure}, and {\em characteristic function} of a set $E$ in $\R^n$, respectively.  The symbols $\aff E$ and $\lin E$ denote respectively the \emph{affine hull} and the \emph{linear hull} of $E$. The \emph{dimension} $\dim E$ is the dimension of $\aff E$. 
%The \emph{relative boundary} of $E$ is the boundary of $E$ in the topology induced on $\aff E$ by the euclidean topology in $\R^n$. 
With the symbol $E\mathbin|L$ we denote the orthogonal projection of $E$ on the linear space  $L$.

Given a function $v$ defined on a subset of $\R^n$, $D v$ and $D^2 v$ denote its gradient and its Hessian, respectively. 

Our next goal is to recall a standard notion of generalised solution of Monge-Amp\`ere equations, usually referred to as
{\em solution in the sense of measure}. Our general reference
for notions and facts about Monge-Amp\`ere equations is the survey by Trudinger and Wang~\cite{TrWa}.
Let $v$ be a convex function defined in an open convex set $\Omega$; the subgradient $\pa v(x)$ of $v$ at $x\in\Omega$ is defined as
\[
\pa v(x) =\{z\in\R^n : v(y)\geq v(x)+\langle z,y-x\rangle \text{ for each $y\in\Omega$}\},
\]
which is a non-empty compact convex set. Note that $v$ is differentiable at $x\in\Omega$ if and only if $\partial v(x)$ consists
 of exactly one vector, which is the gradient of $v$ at $x$.
If $\omega\subset\Omega$ is a Borel set, then we denote by $N_v(\omega)$ the image of $\omega$ through the gradient map of $v$, i.e.
\[
N_v(\omega)=\bigcup_{x\in\omega}\pa v (x). 
\]
Note that as $\omega$ is a Borel set, then $N_v(\omega)$ is measurable. Hence, we may define the Monge-Amp\`ere measure associated to 
$v$ as follows 
\begin{equation}\label{Monge-Ampere-measure}
\mu_v(\omega)=\mathcal{H}^{n}\Big(N_v\big(\omega\big)\Big).
\end{equation}
For $p<1$ and non-negative $g$ on $\R^n$, we say that the non-negative convex function $v$ satisfies the Monge-Amp\`ere equation
$$
v^{1-p}\det( D^2 v )=g
$$
in the sense of measure (or in the Alexandrov sense) if 
$$
v^{1-p}\,d\mu_v=g\,d\mathcal{H}^n.
$$
Equivalently
$$
\int_\omega v^{1-p}(x) d\mu_v(x)=\int_\omega g(x)dx
$$
for every Borel subset $\omega$ of $\Omega$. 

A {\em convex body} in $\R^n$ is a compact convex set with nonempty interior. The treatises 
Gardner \cite{GAR1}, Gruber \cite{PM} and Schneider \cite{SCH} are excellent general references for convex geometry.
The function
$$
h_K(u)=\max\{\left<u,y\right>: y\in K\},
$$
for $u\in\R^n$, is the {\it support function} of $K$. When it is clear the convex body to which we refer  we will drop the subscript $K$ from $h_K$ and write simply $h$.  Any convex body $K$ is uniquely determined by its support function. A set $C\subset\R^n$ is a \emph{convex cone} if  $\alpha_1u_1+\alpha_2u_2\in C$ for any $u_1,u_2\in C$ and $\alpha_1,\alpha_2\geq 0$.

If $S$ is a convex  set in $\R^n$, then $z\in S$ is an extremal point if $z=\alpha x_1+(1-\alpha)x_2$ for
$x_1,x_2\in S$ and $\alpha\in(0,1)$ imply $x_1=x_2=z$. We note that if $S$ is compact and convex, then $S$ is the convex hull of its extremal points.  If $C$ is a convex cone and $u\in C\backslash\{o\}$, we say that 
$\sigma=\{\lambda u:\,\lambda\geq 0\}$ is an extremal ray if $\alpha_1 x_1+\alpha_2x_2\in\sigma$ for
$x_1,x_2\in C$ and $\alpha_1,\alpha_2>0$ imply $x_1,x_2\in\sigma$. Now if $C\neq\{o\}$ is a closed convex cone such that the origin is an extremal point of $C$, then $C$ is the convex hull of its extremal rays.

The \emph{normal cone} of a convex body $K$ at $z\in K$ is defined as 
\[
N(K,z)=\{u\in\R^n : \langle u, y\rangle\leq \langle u, z\rangle\text{ for all $y\in K$}\}
\]
where $N(K,z)=\{o\}$ if $z\in {\rm int} K$ and ${\rm dim}\,N(K,z)\geq 1$ if $z\in\pa K$.
This definition can be written also as 
\begin{equation}
\label{duality_body_support1}
 N(K,z)=\{u\in\R^n : h_K(u)=\langle z,u\rangle\}.
\end{equation}
In particular, $N(K,z)$ is a closed convex cone such that the origin is an extremal point, and
\begin{equation}
\label{notstrictlyconvex}
h_K(\alpha_1u_1+\alpha_2u_2)=\alpha_1h_K(u_1)+\alpha_2h_K(u_2)
\text{\ for $u_1,u_2\in N(K,z)$ and $\alpha_1,\alpha_2>0$.}
\end{equation}
A convex body $K$ is $C^1$-smooth at $p\in\pa K$ if $N(K,p)$ is a ray, and $\partial K$ is $C^1$ if each $p\in\pa K$ is a $C^1$-smooth point. Therefore, $\pa K$ is $C^1$ if and only if the restriction of $h_K$ to any hyperplane not containing $o$ is strictly convex,
by \eqref{notstrictlyconvex}. 

We say that a convex body $K$ is \emph{strictly convex} if $\pa K$ contains no segment. % or equivalently, $h_K$ is $C^1$ on $\R^n\backslash\{o\}$ (see \eqref{duality_body_support2}).
The \emph{face} of $K$ with outer normal 
$u\in\R^n$ is defined as 
\[
 F(K,u)=\{z\in K : h_K(u)=\langle z,u\rangle\},
\]
which lies in $\pa K$ if $u\neq o$.
Schneider~\cite[Thm.~1.7.4]{SCH} proves that
\begin{equation}\label{duality_body_support2}
\pa h_K(u)=F(K,u).
\end{equation}
Therefore, $K$ is strictly convex  if and only if $h_K$ is $C^1$ on $\R^n\backslash\{o\}$.

A crucial notion for this paper is the one of {\em surface area measure} $S_K$ of a convex body $K$, which is a Borel measure
on $S^{n-1}$, defined as follows. For any Borel set $\omega\subset S^{n-1}$:
$$
S_K(\omega)=\mathcal{H}^{n-1}\big(\cup_{u\in\omega} F(K,u)\big)
=\mathcal{H}^{n-1}\big(\cup_{u\in\omega} \pa h_K(u)\big),
$$
Hence, $S_K$ is the analogue of the Monge-Amp\`ere measure for the restriction of $h_K$ to $S^{n-1}$.

Given a convex body $K$ containing $o$ and $p<1$,  let $S_{K,p}$ denote  the {$L_p$ area measure} of $K$; namely,
\begin{equation}\label{p_area_measure}
 d S_{K,p}=h_K^{1-p}d S_K.
\end{equation}

Let $f$ be a positive and measurable function on $S^{n-1}$; we say that $h_K$ is
a solution of \eqref{MongeAmpere_Sn} in the sense of measure if
\begin{equation}\label{new 0}
\int_{\omega} h_K(y)^{1-p}dS_K(y)=
\int_\omega f(y)d\mathcal{H}^{n-1}(y)
\end{equation}
for every Borel subset $\omega$ of $S^{n-1}$. 

In what follows we will always assume that $f$ is bounded between two positive constants. Our first remark is that the 
previous definition is equivalent to the following conditions~(a) and~(b):
\begin{enumerate}[(a)]
\item $\dim N(K,o)<n$; or equivalently,
\begin{equation}\label{sol_alex_a}
\mathcal{H}^{n-1}\big(\{y\in S^{n-1} : h_K(y)=0\}\big)=\mathcal{H}^{n-1}\big(N(K,o)\cap S^{n-1}\big)=0,
\end{equation}
\item for each Borel set $\omega\subset \{y\in S^{n-1}:\,h_K(y)>0\}$, we have
\begin{equation}\label{sol_alex_b first}
\int_\omega h_K^{1-p}(y)\,dS_K(y)
=\int_{\omega} f(y)\, d\mathcal{H}^{n-1}(y).
\end{equation}
\end{enumerate}
Moreover, condition (b) is in turn equivalent to 
\begin{enumerate}[(a)]
\item[(b')] for each Borel set $\omega\subset \{y\in S^{n-1}:\,h_K(y)>0\}$, we have
\begin{equation}\label{sol_alex_b}
S_K(\omega)
=\int_{\omega} f(y)h_K(y)^{p-1}\, d\mathcal{H}^{n-1}(y).
\end{equation}
\end{enumerate}
To prove that (b) and (b') are equivalent is a simple exercise (in which one has to take into account the fact that $h_K$ is continuous).
Indeed, both claims are in turn equivalent to the following fact: the measure $S_K$ is absolutely continuous with respect to 
$\mathcal{H}^{n-1}$ on $S^{n-1}\setminus\{y\in S^{n-1}\colon h_K(y)=0\}$, and the Radon-Nikodym derivative of $S_K$ with respect
to $\mathcal{H}^{n-1}$ is $f h_K^{p-1}$.

Let us prove the equivalence between \eqref{new 0} and (a)-(b). To this end, it will be useful the following observation: the set
$$
\{x\in\R^n\colon h_K(x)=0\}
$$
is a closed convex cone. Indeed, it is the set where the non-negative, convex and 1-homogeneous function $h_k$ attains its minimum. 
For convenience, we set $\omega_0=\{y\in S^{n-1}\colon h_K(y)=0\}$.
Assume that \eqref{new 0} holds; then (b) follows immediately. If, by contradiction, (a) is false, then 
$\omega_0$ has non-empty interior so that
$$
0=\int_{\omega_0} h_K(y)^{1-p} \, dS_K(y)=
\int_{\omega_0} f(y)\, d\mathcal{H}^{n-1}(y)>0,
$$
i.e. a contradiction (in the last inequality we have used the fact that $f$ is bounded from below by a positive constant). 
Vice versa, assume that (a) and (b) hold. Given a Borel subset $\omega$ of $S^{n-1}$ we may write it as the disjoint union of
$\omega'=\omega\cap\omega_0$ and $\omega''=\omega\setminus\omega'$. By (a), $\mathcal{H}^{n-1}(\omega')=0$, moreover
$h_K=0$ on $\omega'$; hence,
\begin{eqnarray*}
\int_\omega h_K(y)^{1-p}\, dS_K(y)&=&\int_{\omega''} h_K(y)^{1-p}\, dS_K(y)\\
&=&\int_{\omega''} f(y)\, d\mathcal{H}^{n-1}(y)\\
&=&\int_{\omega} f(y)\, d\mathcal{H}^{n-1}(y),
\end{eqnarray*}
i.e. \eqref{new 0}.

\medskip

Our next step is to compare the solutions considered by Chou and Wang~\cite{CW} with the ones introduced here.
In particular, we will show that if $h_K$ is a solution  of \eqref{MongeAmpere_CW}, then it verifies conditions (a) and (b) as well
(and consequently \eqref{new 0}). 
Note that being a solution of \eqref{MongeAmpere_CW} in the sense of measures means that
\begin{equation}\label{new}
S_K(\omega)
=\int_{\omega} f(y)h_K(y)^{p-1}\, d\mathcal{H}^{n-1}(y).
\end{equation}
has to hold for every Borel subset of $S^{n-1}$. In particular \eqref{sol_alex_b} follows (and then (b)). Moreover, as $S_K$ is finite, $h_K\ge0$ and $f$ is bounded between two positive constants, the previous relation implies that 
$$
\int_{S^{n-1}}h_K(y)^{p-1}\, d\mathcal{H}^{n-1}(y)<+\infty.
$$
As $p-1<0$, this yields that the set $\omega_0$ where $h_K$ vanishes on $S^{n-1}$ has zero $(n-1)$-dimensional measure. On the other hand 
this is the intersection of $S^{n-1}$ with a convex cone. Hence we get condition (a). 

In addition, if we now apply \eqref{new} to $\omega_0$, we get that when $h_K$ is a solution of~\eqref{MongeAmpere_CW} then
\begin{equation}\label{sol_alex_c}
S_K\big(N(K,o)\cap S^{n-1}\big)=0.
\end{equation}

Note that \eqref{sol_alex_c} implies that $\mathcal{H}^{n-1}(X_0)=0$, in the notation of Theorem~\ref{th_regularity-new}, because $X_0\subset \cup \{F(K,u) : u\in N(K,o)\cap S^{n-1}\}$ and~\eqref{sol_alex_c} means, by definition,
\[
\mathcal{H}^{n-1}\big(\cup_{u\in N(K,o)\cap S^{n-1}} F(K,u)\big)=0.
\]
 Hence,
applying Theorem~\ref{th_regularity-new}~\eqref{th_regularity_f_} we deduce that
if $K\in{\mathcal K}_{0}^n$ is a solution of \eqref{MongeAmpere_CW} for $p<1$ and $f$ is bounded from below and above by positive constants, 
then $\partial K$ is $C^1$, as it was verified by Chou and Wang~\cite{CW}. 

\section{Some results on Monge-Amp\`ere equations in Euclidean space} 
\label{sec-Monge-Ampere}

Lemma~\ref{MongeAmpereRn-lemma} is the tool to transfer the Monge-Amp\`ere equation (\ref{MongeAmpere_Sn})
on $S^{n-1}$ to a Euclidean Monge-Amp\`ere equation on $\R^{n-1}$.
For $e\in S^{n-1}$, we consider the restriction of a solution $h$ of \eqref{MongeAmpere_Sn} to the  hyperplane tangent to $S^{n-1}$ at  $e$. 

\begin{lemma}
\label{MongeAmpereRn-lemma}
If $e\in S^{n-1}$, $h$ is a convex positively $1$-homogeneous non-negative function on $\R^n$ that is a solution of \eqref{MongeAmpere_Sn} for $p<1$ and positive $f$,  and $v(y)=h(y+e)$ holds for $v:\,e^\bot\to \R$, then 
$v$ satisfies
\begin{equation}\label{MongeAmpereRn}
v^{1-p}\det( D^2 v )=g \quad \text{ on $e^\bot$}
\end{equation}
where, for $y\in  e^\bot$, we have
\[
g(y)=\left(1+\|y\|^2\right)^{-\frac{n+p}2} f\left(\frac{e+y}{\sqrt{1+\|y\|^2}}\right).
\]
\end{lemma}
\begin{proof}
Let $h=h_K$ for $K\in\mathcal{K}_0^n$, and let
$$
\widetilde{S}=\{u\in S^{n-1}:\,h_K(u)=0\},
$$
which is a possibly empty spherically convex compact set whose spherical dimension is at most $n-2$,
by (\ref{sol_alex_a}). According to (\ref{sol_alex_b}), the Monge-Amp\`ere equation for $h_K$ can be written in the form
\begin{equation}
\label{hKnozero}
 dS_K=h_K^{p-1}f\,d \mathcal{H}^{n-1}\mbox{ \ \ on $S^{n-1}\backslash \widetilde{S}$}.
\end{equation}

We consider $\pi:e^\bot\to S^{n-1}$ defined by
$$
\pi(x)=(1+\|x\|^2)^{\frac{-1}2}(x+e),
$$
which is induced by the radial projection from the tangent hyperplane $e+e^\bot$ to $S^{n-1}$. 
Since $\langle \pi(x),e\rangle=(1+\|x\|^2)^{\frac{-1}2}$, the Jacobian of $\pi$ is
\begin{equation}
\label{Dpi}
\det D\pi(x)=(1+\|x\|^2)^{\frac{-n}2}.
\end{equation}

For $x\in e^\bot$, (\ref{duality_body_support2}) and writing $h_K$ in terms of an orthonormal basis of $\R^n$ containing $e$,  
yield that $v$ satisfies
$$
\pa v(x)=\pa h_K(x+e)|e^\bot=F(K,x+e)|e^\bot=F(K,\pi(x))|e^\bot.
$$
Let $S=\pi^{-1}(\widetilde{S})$. 
For a Borel set $\omega\subset e^\bot\backslash S$, we have
\begin{eqnarray*}
\mathcal{H}^{n-1}(N_v(\omega))&=&
\mathcal{H}^{n-1}\left(\cup_{x\in\omega}\pa v(x)\right)\\
&=&\mathcal{H}^{n-1}\left(\cup_{u\in\pi(\omega)}\left(F(K,u)|e^\bot\right)\right)
=\int_{\pi(\omega)}\langle u,e\rangle\,dS_K(u)\\
&=&\int_{\pi(\omega)}\langle u,e\rangle h_K^{p-1}(u)f(u)\,d \mathcal{H}^{n-1}(u)\\
&=&\int_\omega(1+\|x\|^2)^{\frac{-n-p}2} f(\pi(x))v(x)^{p-1}\,d \mathcal{H}^{n-1}(x)
\end{eqnarray*}
where we used at the last step that
$$
v(x)=h_K(x+e)=(1+\|x\|^2)^{\frac{1}2}h_K(\pi(x)).
$$
In particular, $v$ satisfies the Monge-Amp\`ere type differential equation
$$
\det D^2v(x)=(1+\|x\|^2)^{\frac{-n-p}2} f(\pi(x))v(x)^{p-1}
\mbox{ \ \ on $e^\bot\backslash S$}.
$$
Since ${\rm dim}\,S\leq n-2$ by \eqref{NKosmall}, $v$ satisfies (\ref{MongeAmpereRn}) on $e^\bot$.
\end{proof}

Having Lemma~\ref{MongeAmpereRn-lemma} at hand showing the need to understand related Monge-Amp\`ere equations in Euclidean spaces, we prove Propositions~\ref{vanishingonsegment}
and \ref{ustrictconvex}, and quote Caffarelli's Theorem~\ref{Caffarelli-smooth}.

\begin{proof}[Proof of Proposition~\ref{vanishingonsegment}]
Up to changing coordinate system, we may assume, without loss of generality, that
$S\subset \{(x_1,x_2)\in\R^r\times\R^{n-r}: x_2=0\}$ and the origin is contained in the relative interior of $S$.
Therefore, up to restricting $\Omega$, we may also assume that 
$v$ is continuous on ${\rm cl}\,\Omega$, that
$\Omega=\{(x_1,x_2)\in\R^r\times\R^{n-r}: \|x_1\|< s_1, \|x_2\|< s_2\}$ for some constants $s_1,s_2>0$ and that $S=\{(x_1,x_2)\in\Omega : x_2=0\}$.

Let  $\alpha=\max_{{\rm cl}\,\Omega} v$ and let us consider the convex body
\[
M=\{(x_1,x_2,y)\in \R^r\times\R^{n-r}\times\R : \|x_1\|\leq s_1, \|x_2\|\leq s_2, v(x_1,x_2)\leq y\leq \alpha\}.
\]
For $t\in(0,s_2/2]$, let
\[
\Omega_t=\{(x_1,x_2)\in \R^r\times\R^{n-r} : \|x_1\|\leq s_1/2, \|x_2\|\leq t\}.
\]

We estimate $\mathcal{H}^{n}\big(N_v(\Omega_t\setminus S)\big)$. Let $(x_1,x_2)\in \Omega_t\setminus S$ and let $(z_1,z_2)\in \R^r\times\R^{n-r}$ belong to $\pa v (x_1,x_2)$. We prove that
\begin{equation}\label{bound_subgradient}
 \|z_2\|\leq \frac{2\alpha}{s_2}\quad\text{and}\quad \|z_1\|\leq\frac{4\alpha}{s_1 s_2} t.
\end{equation}
If $z_2=0$ the first inequality in \eqref{bound_subgradient} holds true. Assume $z_2\neq0$.
The vector $(z_1,z_2,-1)$ is an exterior normal to $M$ at $p=(x_1,x_2,v(x_1,x_2))$. Since 
\[
q_1=\left(x_1,x_2+\frac{s_2 z_2}{2\|z_2\|},\alpha\right)\in M
\] 
(because  $\big\|x_2+s_2 z_2/(2\|z_2\|)\big\|\leq \|x_2\|+s_2/2\leq s_2$) then $\langle q_1-p, (z_1,z_2,-1)\rangle\leq0$. This implies
\[
\|z_2\|\leq\frac{2}{s_2}(\alpha-v(x_1,x_2))
\]
and the first inequality in \eqref{bound_subgradient}. Again, if $z_1=0$ then the second inequality \eqref{bound_subgradient} holds true. Assume $z_1\neq0$. We have
\[
q_2=\left(x_1+\frac{s_1 z_1}{2 \|z_1\|}, 0, v(x_1,x_2)\right)\in M,
\]
because $\big\|x_1+s_1 z_1/(2\|z_1\|)\big\|\leq s_1$,  $(x_1+s_1 z_1/(2\|z_1\|),0)\in S$ and therefore $v(x_1,x_2)\geq 0=v(x_1+s_1 z_1/(2\|z_1\|),0)$.
The inequality $\langle q_2-p, (z_1,z_2,-1)\rangle\leq0$ implies the second inequality \eqref{bound_subgradient}.

The inequalities in \eqref{bound_subgradient} imply
\begin{equation}\label{th_regularity_c_estimate_Omegat}
 \mathcal{H}^{n}\big(N_v(\Omega_t\setminus S)\big)\leq c\ t^r,
\end{equation}
for a suitable constant $c$ independent of $t$.

Now we estimate $\int_{\Omega_t\backslash S} v(x)^{p-1}\ dx$. 
The inclusion of the convex hull of $S\times\{0\}$ and 
$\{\|x_1\|\leq s_1, \|x_2\|\leq s_2, y=\alpha\}$ in $M$ implies that $v(x_1,x_2)\leq\frac{\alpha}{s_2}\, \|x_2\|$ for each $(x_1,x_2)\in\Omega_t$ by the convexity of $v$. Using this estimate it is straightforward to compute that 
\begin{equation}\label{th_regularity_c_estimate_integral}
 \int_{\Omega_t\backslash S} v(x)^{p-1}\ dx\geq d\ t^{n+p-r-1},
\end{equation}
for a suitable constant $d$ independent on $t$. 
The inequalities \eqref{th_regularity_c_estimate_Omegat} and \eqref{th_regularity_c_estimate_integral} and the differential inequality satisfied by $v$  imply, as $t\to0^+$, 
\[
 c t^r\geq \mathcal{H}^{n}\big(N_v(\Omega_t\setminus S)\big)\geq \int_{\Omega_t\backslash S} \tau v(x)^{p-1}\ dx\geq \tau d\ t^{n+p-r-1}.
\]
This inequality implies $p\geq -n+1+2r$.
\end{proof}

% \begin{example}\label{example_vanishingonsegment}The inequality $p\geq-n+1+2r$ is optimal, at least when $r=1$. Indeed for any $p>-n+3$ there exists   a non-negative convex solution of \eqref{MongeAmpereOmega} in $\Omega$ which vanish on a segment $S$.  

\begin{example}\label{example_vanishingonsegment}Let us show that for any $p>-n+3$ there exists   a non-negative convex solution of \eqref{MongeAmpereOmega} in $\Omega=\{(x_1,x_2)\in\R\times\R^{n-1}: x_1\in[-1,1], \|x_2\|\leq 1\}$ which vanish on the $1$-dimensional space $S=\{(x_1,x_2)\in\R\times\R^{n-1} : x_2=0\}$.  

To prove this let
\[
v(x_1,x_2) = \|x_2\|+f(\|x_2\|)g(x_1)
\]
where $f(r)=r^\alpha$, with $\alpha=(p+n-1)/2$, and $g(x_1)=(1+\beta x_1^2)$, with $\beta>0$ sufficiently small. Note that $\alpha>1$ exactly when $p>-n+3$.

% My computations show that (modulo possible mistakes)
% \begin{align*}
%  u_{yy}&=f(|x|) g''(y)\\
%  u_{yx_i}&=f'(|x|)\frac{x_i}{|x|} g'(y)\\
%  u_{x_i x_j}&=\frac{\delta_{ij}|x|^2-x_ix_j}{|x|^3}+\Big(f''(|x|)\frac{x_ix_j}{|x|^2}+f'(|x|)\frac{\delta_{ij}|x|^2-x_ix_j}{|x|^3}\Big) g(y)
% \end{align*}
The function $v$ is invariant with respect to rotations around the line containing $S$. To compute $\det D^2 v$ 
at an arbitrary point, it suffices to compute it at $(x_1,0,\dots,0,r)$, $r\ge0$. We get
\begin{align*}
 &v_{x_1x_1}=f(r) g''(x_1), & &\\
 &v_{x_1x_i}=0 & &\text{when $1<i<n$,}\\
 &v_{x_1x_n}=f'(r) g'(x_1), & &\\
 &v_{x_ix_i}=\frac1{r}+\frac{f'(r)}{r} g(x_1) & &\text{when $1<i<n$,}\\
 &v_{x_i x_j}=0 & &\text{when $i\neq j$, $(i,j)\neq(1,n)$, $(i,j)\neq(n,1)$,}\\
 &v_{x_nx_n}=f''(r)g(x_1). & &
\end{align*}
The function $v$ is convex if $\beta$ is sufficiently small. Indeed, the eigenvalues of $D^2 v$ are $\frac1{r}+\frac{f'(r)}{r} g(x_1)$, with multiplicity $n-2$, and those of the matrix
\[
\left(
\begin{array}{ll}
f g'' & f'g'\\
f'g' & f''g
 \end{array}
 \right).
\]
The determinant of the latter matrix is 
\[
 2\alpha\beta r^{2(\alpha-1)}\Big(\alpha-1-(1+\alpha)\beta x_1^2\Big),
\]
which is positive if $\beta>0$ is sufficiently small. Thus, all eigenvalues of $D^2 v$ are positive.

We get
\[
\det D^2 v= \Big( f'' g f g'' -( f' g')^2 \Big)  \Big(\frac1{r}+\frac{f'}{r} g\Big)^{n-2}
\]
which has the same order as $r^{2\alpha -n}$ as $r\to 0^+$. Clearly $v$ has order $r$, and 
$v^{1-p}\det D^2 v$ has order $r^{2\alpha-n+1-p}$, which is uniformly bounded from above and below 
for our choice of $\alpha$.
\end{example}

The next statement is a slight modification of Lemmas~3.2 and 3.3 from Trudinger and Wang~\cite{TrWa}. Its proof closely follows that in \cite{TrWa} and is given here for completeness.
%We remark that Lemma~3.2 in~\cite{TrWa} proves~\eqref{formula_lemmaTW_2} with $\sup_\Omega|v|$ instead of $|v(o)|$. The inequality~\eqref{formula_lemmaTW_2} follows from that and the observation  that  if $u$ is any  convex function in $\Omega$, which vanishes on $\partial\Omega$, and  $tE\subset \Omega\subset E$ then $|u(o)|\geq t/(t+1) \sup_\Omega |u|$.

\begin{lemma}\label{MongeAmperepointinside}
Let $v$ be a convex function defined on the closure of an open bounded convex set $\Omega\subset \R^n$ satisfying the 
Monge-Amp\`ere equation
$$
\det D^2 v=\nu
$$
for a finite non-negative measure $\nu$ on $\Omega$, let $v\equiv 0$ on $\partial \Omega$ and let $tE\subset\Omega\subset E$ for $t>0$ and an origin centred ellipsoid $E$.
\begin{enumerate}[(i)]
\item\label{MApi_i} If $z\in\Omega$ satisfies $(z+s\, E)\cap\partial\Omega\neq \emptyset$ for $s>0$, then
$$
|v(z)|\leq s^{1/n}c_0\mathcal{H}^n(\Omega)^{1/n}\nu(\Omega)^{1/n}
$$
for some $c_0>0$ depending on $n,t$.
\item\label{MApi_ii} If  $\nu(t\Omega)\geq b\,\nu(\Omega)$ for $b>0$, then
\begin{equation}\label{formula_lemmaTW_2}
|v(0)|\geq c_1\mathcal{H}^n(\Omega)^{1/n}\nu(\Omega)^{1/n}
\end{equation}
for some $c_1>0$ depending on $n$, $t$ and $b$.
\item\label{MApi_iii} If  $(z+s\, E)\cap\partial\Omega\neq \emptyset$ and $\nu(t\Omega)\geq b\,\nu(\Omega)$ then
\begin{equation}\label{formula_lemmaTW_3}
\frac{|v(z)|}{|v(o)|}\leq \frac{c_1}{c_0}s^{1/n}.
\end{equation}
\end{enumerate}
\end{lemma}
When $E=B^n$ the number $s$ can be chosen as the distance of $z$ from $\pa\Om$. In the general case $s$ has the same meaning in the metric induced by the norm whose unitary ball is $E$.
\begin{proof}
 Let $A$ be a linear transformation such that $B^n=A^{-1}E$, let $\tilde{v}(x)=v(Ax)|\det A|^{-2/n}$, $\widetilde{\Om}=A^{-1}\Om$ and let $\tilde{\nu}$ be the measure 
 defined for each Borel set $\om\subset\widetilde{\Om}$ as $\tilde{\nu}(\om)=\nu(A\om)/|\det A|$. It is known that $\tilde{v}$ solves 
 \begin{equation}\label{equation'}
  \det D^2 \tilde{v}=\tilde{\nu}\quad\text{ in $\widetilde{\Om}$.}
 \end{equation}
Moreover, $t B^n\subset\widetilde{\Om}\subset B^n$. Since $\mathcal{H}^n(\Om)=|\det A|\mathcal{H}^n(\widetilde{\Om})$, we have
\begin{equation}\label{det_A}
 \frac{\mathcal{H}^n(\Om)}{\om_n}\leq|\det A|\leq \frac{\mathcal{H}^n(\Om)}{\om_n t^n}.
\end{equation}

Let us prove Claim~\eqref{MApi_i}. Let $\tilde{z}=A^{-1}z$. Then $(\tilde{z}+s B^n)\cap\pa\widetilde{\Om}\neq\emptyset$ and if $d$ denotes the distance of $\tilde{z}$ from $\pa\widetilde{\Om}$ we have $d\leq s$. By choosing proper coordinates we may assume that $\tilde{z}=(0,\dots,0,d)$, and that $\widetilde{\Om}\subset\{(x_1,\dots,x_n)\in\R^n : x_n>0\}$. Then 
\[
\widetilde{\Om}\subset\widehat{\Om}=\{(x_1,\dots,x_n)\in\R^n : \|(x_1,\dots,x_{n-1})\|<2, 0< x_n<4\}.
\]
 Let $u$ and $w$ be convex functions such that their graphs are convex cones with vertex at $(\tilde{z},\tilde{v}(\tilde{z}))$ and bases $\pa \widetilde{\Om}$ and $\pa\widehat{\Om}$, respectively. Then
 \begin{equation}\label{cones_inclusion1}
  N_{\tilde{v}}(\widetilde{\Om})\supset N_{u}(\widetilde{\Om})=\pa u(\tilde{z})\supset\pa w(\tilde{z}).
 \end{equation}
Since $w$ is a convex cone over the cylinder $\widehat{\Om}$, one can easily compute that  $\mathcal{H}^n(\pa w(\tilde{z}))\geq c_2 |\tilde{v}(\tilde{z})|^n/d$, for a suitable constant $c_2>0$. This inequality, \eqref{equation'} and \eqref{cones_inclusion1}  imply
\[
 |\tilde{v}(\tilde{z})|\leq\left(\frac{d}{c_2}\right)^{1/n}\mathcal{H}^n(N_{\tilde{v}}(\widetilde{\Om}))^{1/n}=
 \left(\frac{d}{c_2}\right)^{1/n}\tilde{\nu}(\widetilde{\Om})^{1/n}.
\]
Expressing this inequality in terms of $v$, $\Om$ and $\nu$ and using $d\leq s$ and \eqref{det_A} concludes the proof of Claim~\eqref{MApi_i}.

Let us prove Claim~\eqref{MApi_ii}. 
There exists an unique solution $w$ of $\det D^2 w=\widehat{\nu}$ in $\widetilde{\Om}$, $w=0$ in $\pa\widetilde{\Om}$, where $\widehat{\nu}=\tilde{\nu}$ in $t\widetilde{\Om}$ and $\widehat{\nu}=0$ elsewhere (see Theorem 2.1 in \cite{TrWa}).  The comparison principle for Monge-Amp\`ere equations (see Lemma 2.4 in \cite{TrWa}) implies $w\geq \tilde{v}$ in $\widetilde{\Om}$. 

Let $z\in t\widetilde{\Om}$. The distance  $d$ of $z$ from $\pa\widetilde{\Om}$ is larger than or equal to  $(1-t)t$ (here we have used the inclusion $tB^n\subset\widetilde{\Om}$). If $y\in\pa w (z)$ and $l(x)=\left<x,y\right>+w(z)$ then $l(x)\leq w(x)$ for each $x\in\widetilde{\Om}$, by definition of subgradient. In particular, we have $l(x)\leq0$ for each $x\in\pa\widetilde{\Om}$. This implies
\[
|y|\leq\frac{|w(z)|}{d}\leq \frac{\sup_{\widetilde{\Om}} |\tilde{v}|}{t(1-t)}.
\]
Therefore,
\[
 \mathcal{H}^n(N_w(t\widetilde{\Om}))\leq\om_n\left(\frac{\sup_{\widetilde{\Om}} |\tilde{v}|}{t(1-t)}\right)^n.
\]
This inequality, the equation satisfied by $w$ and the condition $\nu(t\Omega)\geq b\,\nu(\Omega)$ imply
\begin{equation}\label{ineq_intermed}\begin{aligned}
 \sup_{\widetilde{\Om}} |\tilde{v}|\geq&\frac{t(1-t)}{\om_n^{1/n}}\mathcal{H}^n(N_w(t\widetilde{\Om}))^{1/n}=
 \frac{t(1-t)}{\om_n^{1/n}}\tilde{\nu}(t\widetilde{\Om})^{1/n}\\
 \geq& \frac{b t(1-t)}{\om_n^{1/n}}\tilde{\nu}(\widetilde{\Om})^{1/n}.
\end{aligned}\end{equation}
We claim that
\begin{equation}\label{relation_sup_0}
 |\tilde{v}(o)|\geq\frac{t}{1+t}\sup_{\widetilde{\Om}}|\tilde{v}|.
\end{equation}
Indeed, let $z\in\widetilde{\Om}$ be such that $\tilde{v}(z)=\inf_{\widetilde{\Om}}\tilde{v}$. We may clearly assume $z\neq0$, since otherwise there is nothing to prove. By choosing proper coordinates we may assume $z=(z_1,0,\dots,0)$ for some $z_1>0$. 
Let $l$ be the linear function defined on the line through $o$ and $z$ and such that $l(o)=\tilde{v}(o)$ and $l(z)=\tilde{v}(z)$. 
It is $l(s,0,\dots0)=\tilde{v}(o)+s(\inf_{\widetilde{\Om}}\tilde{v}-\tilde{v}(o))/z_1$. Since $\tilde{v}$ is convex,
\[
 l(s,0,\dots0)\leq \tilde{v}(s,0,\dots0)
\]
for each $s\notin[0,z_1]$ such that $(s,0,\dots,0)\in\widetilde{\Om}$ . When $s=-t$ we obtain $l(-t,0,\dots,0)\leq \tilde{v}(-t,0,\dots,0)\leq0$. The inequality $l(-t,0,\dots,0)\leq0$ and the inclusion $\widetilde{\Om}\subset B^n$ imply \eqref{relation_sup_0}.

The proof of Claim~\eqref{MApi_ii} is concluded by combining \eqref{ineq_intermed} and \eqref{relation_sup_0} and expressing the obtained inequality in terms of $v$, $\Om$ and $\nu$.

Claim~\eqref{MApi_iii} is a consequence of the first two claims.
\end{proof}

The proof of Claim~\eqref{th_regularity_f_} in Theorem~\ref{th_regularity-new} is based on the following proposition, which is related to a step in the proof of Theorem~E~(a) in~\cite{CW}, however our proof is substantially different from that in \cite{CW}.

\begin{prop}
\label{ustrictconvex}
Let $v$ be a non-negative convex function defined on the closure of an open convex set $\Omega\subset \R^n$, 
$n\geq 2$,
such that $S=\{x\in \Omega:v(x)=0\}$ is non-empty and compact, and $v$ is locally strictly convex on $\Omega\backslash S$.
Let $\psi:(0,\infty)\to[0,\infty)$ be monotone decreasing and not identically zero; assume that $\tau_2>\tau_1>0$ and $v$ satisfy
\begin{equation}\label{utau1tau2}
\tau_1\psi(v)\leq \det D^2 v\leq \tau_2\psi(v)
\end{equation}
in the sense of measure on $\Omega\backslash S$. If ${\rm dim}\,S\leq n-1$ and $\mu_v(S)=0$ for the associated Monge-Amp\`ere measure $\mu_v$, then $S$ is a point.
\end{prop}

Note that~\eqref{utau1tau2} means that for each Borel set $\omega\subset \Omega\setminus S$ we have
\[
 \tau_1\int_\omega\psi(v(x))\,dx\leq \mu_v(\omega)\leq \tau_2\int_\omega\psi(v(x))\,dx,
\]
where $\mu_v$ has been defined in~\eqref{Monge-Ampere-measure}.
\begin{proof}
We  assume, arguing by contradiction that $S$ is not a point.
Choose coordinates so that $o$ is the centre of mass of $S$.  Let $L=\lin \, S$. By assumption
\begin{equation}\label{definition_m}
1\leq\dim L\leq n-1.
\end{equation}
Let $e=(o,1)\in\R^n\times\R$. We may assume that $\Omega$ is bounded, after possibly substituting it with a bounded open neighbourhood of $S$.
We start by illustrating the idea of the proof.

\textit{Sketch of the proof.}
For any small $\varepsilon>0$, we construct an affine function $l_\varepsilon$ such that 
$l_\varepsilon(x)=\varepsilon$ for $x\in L$, and
 the convex set $\Omega_\varepsilon=\{v<l_\varepsilon\}$ is well-balanced; namely,
there exists an ellipsoid $E_\varepsilon$ centred at the origin such that
$(1/(8 n^3))E_\varepsilon\subset \Omega_\varepsilon\subset E_\varepsilon$ (see (\ref{Omegaepsellipsoid})). This is the longest part of the argument, and the main idea to construct $l_\varepsilon$ is that the graph of $l_\varepsilon$ cuts off the smallest volume cap from the graph of $v$ among the hyperplanes in $\R^{n+1}$ containing $L+\varepsilon e$.
Subsequently, we apply Lemma~\ref{MongeAmperepointinside} to $\Omega_\ve$ and to the function $v-l_\varepsilon$ in the standard way to reach a contradiction. We show that one can choose $z\in S$ so that the corresponding parameter $s$, as defined in Lemma~\ref{MongeAmperepointinside}, tends to $0$ as $\veps$ tends to $0$. (Equivalently, $S$ contains points whose distance from $\pa \Om_\veps$, the one induced by the norm whose unit ball is $E_\veps$, tends to $0$ as $\veps$ tends to $0$.) This contradicts  \eqref{formula_lemmaTW_3}, since $|v(z)-l_\veps(z)|/|v(o)-l_\veps(o)|=\veps/\veps=1$.

\medskip

We divide the proof into four steps.

\medskip

\textit{Step 1. Definition of $l_\veps$ and of $\Om_\veps$.}

Let $\varepsilon_0=\min_{\pa\Om} v>0$ and let us consider the $(n+1)$-dimensional convex body 
\[
M=\{(x,y)\in\R^{n}\times\R:v(x)\leq y\leq \varepsilon_0\}.
\]
For $\veps\in(0,\veps_0)$ define $H_\varepsilon$ to be a hyperplane in $\R^{n+1}$
\begin{enumerate}[(i)]
\item containing $L+\varepsilon e=\{(x,\varepsilon e)\in\R^n\times\R\colon x\in L\}$ and
\item cutting off the minimal volume from $M$ (on the side containing the origin) under condition (i).
\end{enumerate}

Let $r>0$. We claim that there exists $\veps_1=\veps_1(r)$ so that  $H_\veps$ is the graph of an affine function $l_\veps$ for each $\varepsilon\in(0,\varepsilon_1)$, and, setting
$$
\Omega_\varepsilon=\{x\in\R^n:v(x)<l_\varepsilon(x)\},
$$
we have
\begin{equation}\label{inclusions_Om_eps}
\cl\Om_\veps\subset\Om,\quad S\subset \Omega_\varepsilon\quad\text{and}\quad \Omega_\varepsilon\cap L\subset (1+r)S.
\end{equation}
Let $F=\{(x,y)\in M : y=\varepsilon_0\}$ be the upper face of $M$ and let $\mathcal {H}$ be the collection of hyperplanes in $\R^{n+1}$  which intersect both $F$ and $\{(x,y)\in M : y\leq\varepsilon_0/2\}$. Since $\Omega$ is bounded and $v$ is locally strictly convex on $\Omega\setminus S$, every hyperplane in $\mathcal {H}$ is not a supporting hyperplane to $M$. 
Therefore, by compactness, there exists a constant $\varrho_0>0$ such that for every $H\in\mathcal {H}$ both components of  $M\backslash H$ are of volume at least $\varrho_0$. 
We choose $\varepsilon_1\in(0,\varepsilon_0/2)$ such that the volume of the cap $\{(x,y)\in M:\,y\leq \varepsilon_1\}$ is less than $\varrho_0$. This choice implies that the minimum value of the problem which defines $H_\veps$ is less than $\varrho_0$. Therefore, a minimizer $H_\veps$ does not belong to $\mathcal{H}$. Since $H_\veps\cap\{(x,y)\in M : y\leq\varepsilon_0/2\}\neq0$, we have $H_\veps\cap F=\emptyset$. In particular, $H_\veps$ is the graph of a affine function defined on $\R^n$ and $\cl \Om_\veps\subset\Om$.

The inclusion $S\subset \Om_\veps$ holds because $v(x)=0$ and $l_\veps(x)=\veps$ for any $x\in S$. 

The origin $o$, being the centre of mass of $S$, belongs to the relative interior of $S$. Since $\dim S>0$, the relative boundary  of $(1+r) S$  does not intersect $S$. This implies $\inf_{\relbd (1+r) S} v>0$. Thus, if $\veps_1$ satisfies
\[
\veps_1<\inf_{\relbd (1+r) S} v
\]
in addition to the inequalities specified above then $v(x)>\veps$ and $l_\veps(x)=\veps$ for any $x\in \relbd (1+r) S$ ($l_\veps(x)=\veps$ is a consequence of $(1+r)S\subset L$). This implies $\Omega_\varepsilon\cap L\subset (1+r)S$.

In the rest of the proof we may assume $\veps_1<\veps_1(1)$ so that
\begin{equation}
\label{SOmegaep2}
\Omega_\varepsilon\cap L\subset 2S.
\end{equation}

\medskip

\textit{Step 2. The centre of mass of $\Omega_\varepsilon$ is contained in $L$.}

To prove this claim we have  to prove that for each $w\in L^\perp\cap\R^n$ we have
\begin{equation}\label{baricenter_onL}
\int_{\Omega_\varepsilon}\langle x,w\rangle\,dx=0.
\end{equation}
Indeed, for $t\in\R$ with $|t|$ small, let 
\[
 F(t)=\int_{\{x\in\Omega: l_\varepsilon(x)+t\langle x,w\rangle-v(x)>0\}} (l_\varepsilon(x)+t\langle x,w\rangle -v(x))\ dx
\]
be the volume cut off by the hyperplane in $\R^{n+1}$ that is the graph of $x\mapsto l_\varepsilon(x)+t\langle x,w\rangle$ from $M$.
By definition of $H_\varepsilon$ and $l_\varepsilon$, $F$ has a local minimum at $t=0$. We have 
\begin{multline*}
  \frac{F(t)-F(0)}{t}=\int_{\{x\in\Omega: l_\varepsilon(x)-v(x)>0\}} \langle x,w\rangle\ dx\\
  +\int_\Omega\Big( \frac{l_\varepsilon(x)-v(x)}{t}+\langle x,w\rangle  \Big)\Big(1_{\{x: l_\varepsilon(x)+t\langle x,w\rangle-v(x)>0\}}-1_{\{x: l_\varepsilon(x)-v(x)>0\}}\Big)\ dx.
\end{multline*}
The set where $1_{\{x: l_\varepsilon(x)+t\langle x,w\rangle-v(x)>0\}}-1_{\{x: l_\varepsilon(x)-v(x)>0\}}$ differs from $0$ is contained in  
\[
A_t=\{x\in\Omega :\left|l_\varepsilon(x)-v(x)\right|<|t \langle x,w\rangle|\}
\]
and there exists $c$ independent on $t$ such that $\mathcal{H}^n(A_t)< c t$ and $\sup_{A_t}|l_\varepsilon(x)-v(x)|< ct$. As $F$ has a local minimum at $t=0$, we have
\[
0= \frac{d F}{dt}(0)=\int_{\Omega_\varepsilon} \langle x,w\rangle\ dx,
\]
which proves \eqref{baricenter_onL}.

\medskip

\textit{Step 3. For any $\varepsilon\in(0,\varepsilon_1)$ there exists an ellipsoid $E_\varepsilon$ centred at the origin such that
\begin{equation}
\label{Omegaepsellipsoid}
\frac1{8 n^3}E_\varepsilon\subset \Omega_\varepsilon\subset E_\varepsilon.
\end{equation}}

Lemma 2.3.3 in \cite{SCH} proves that any $k$-dimensional convex body contains its reflection, with respect to its centre of mass, scaled, with respect to the same centre of mass, by $1/k$. 
From the fact that the centre of mass of $\Om_\veps$ belongs to $L$ we deduce that
\begin{equation}
\label{Omegaepsnproj}
-(\Omega_\varepsilon|L^\bot)\subset n (\Omega_\varepsilon|L^\bot).
\end{equation}

According to Loewner's or John's theorems, there exists an ellipsoid $\widetilde{E}$ centred at the origin and 
$z_1\in\Omega_\varepsilon$ such that
$$
z_1+\frac1n\,\widetilde{E}\subset \Omega_\varepsilon\subset z_1+\widetilde{E}.
$$
It follows from (\ref{Omegaepsnproj}) that there exists $z_2\in\Omega_\varepsilon$ such that
$z_2|L^\bot=\frac{-1}n\,z_1|L^\bot$. In particular, 
$y_1=\frac1{n+1}z_1+\frac{n}{n+1}z_2\in\Omega_\varepsilon$ verifies $y_1|L^\bot=o$, or in other words, 
$y_1\in L\cap \Omega_\varepsilon$. In addition,
$$
y_1+\frac1{2n^2}\,\widetilde{E}\subset \frac1{n+1}
\left(z_1+\frac1n\,\widetilde{E}\right)+\frac{n}{n+1}z_2\subset \Omega_\varepsilon.
$$
Let $m={\rm dim}\,L\leq n-1$. Since $y_1\in L\cap \Omega_\varepsilon$ and (\ref{SOmegaep2}) imply $\frac12\,y_1\in S$, and since the origin is the centroid of $S$, we deduce that $y_2=\frac{-1}{2m}\,y_1\in S$. As $2m+1<2n$, we have
$$
\frac1{4n^3}\,\widetilde{E}\subset \frac1{2m+1}
\left(y_1+\frac1{2n^2}\,\widetilde{E}\right)+\frac{2m}{2m+1}y_2\subset \Omega_\varepsilon.
$$
As $\Omega_\varepsilon\subset 2\widetilde{E}$ follows from $o\in z_1+\widetilde{E}$, we may choose $E_\varepsilon=2\widetilde{E}$, proving (\ref{Omegaepsellipsoid}).

\medskip

\textit{Step 4. Application of Lemma~\ref{MongeAmperepointinside} to $v-l_\veps$ and $\Om_\veps$ and contradiction.}

We observe that
\begin{equation}
\label{vminusl}
v(x)-l_\varepsilon(x)=
\left\{\begin{array}{rl}
0&\mbox{ if $x\in\partial \Omega_\varepsilon$}\\
-\varepsilon& \mbox{ if $x\in S$}.
\end{array} \right.
\end{equation}

Let $\nu$ denote the Monge-Amp\`ere measure $\mu_{(v-l_\varepsilon)}$ restricted to $\Omega_\varepsilon$.  If  $\Omega_0$ is an open set such that $\Omega_\varepsilon\subset \Omega_0\subset{\rm cl}\,\Omega_0\subset \Omega$, then the set $N_v(\Omega_0)$ is bounded and this implies 
\[
\nu(\Omega_\ve)=\mathcal{H}^{n}(N_{(v-l_\varepsilon)}(\Omega_\varepsilon))\leq \mathcal{H}^{n}(N_v(\Omega_0))<\infty.
\]

Let  $t=1/(8 n^3)$. Formula \eqref{Omegaepsellipsoid} yields $tE_\varepsilon\subset\Omega_\ve\subset E_\varepsilon$. 
Let us prove that 
\begin{equation}
\label{bttau}
\nu(t\Omega_\ve)\geq b \nu(\Omega_\ve) \mbox{ \ for $b=\tau_1t^n/\tau_2$.}
\end{equation}
 The function  $v$ is convex and attains its minimum at $o$, thus $v(x)\geq v(tx)$ for any $x\in \Omega_\ve$.
By this fact, the monotonicity of $\psi$, \eqref{utau1tau2} and the assumptions on $S$, we deduce that
\begin{align*}
\nu(t\Omega_\ve)=\nu(t(\Omega_\ve\setminus S))&\geq\tau_1\int_{t(\Omega_\varepsilon\setminus S)}\psi(v(x))\,dx\\
&=\tau_1t^n \int_{\Omega_\varepsilon\setminus S}\psi(v(tz))\,dz\\
&\geq\tau_1t^n \int_{\Omega_\varepsilon\setminus S}\psi(v(z))\,dz\\
&\geq \frac{\tau_1t^n}{\tau_2}\,\nu(\Omega_\varepsilon\setminus S)=\frac{\tau_1t^n}{\tau_2}\,\nu(\Omega_\varepsilon)
\end{align*}
proving (\ref{bttau}).

Let $z\in\relbd S$. We claim that when $\veps\in(0,\veps_1(r))$ then $(z+rE_\veps)\cap\pa\Om_\veps\neq\emptyset$. This is a consequence of the second and third inclusion in~\eqref{inclusions_Om_eps}.
Indeed, since $o\in S\subset\Om_\veps\subset E_\veps$, there exists $q_\veps>0$ such that $(1+q_\veps)z\in\pa E_\veps$. The set $z+r E_\veps$ contains the segment $[z, z+r(1+q_\veps)z]$. Since $q_\veps>0$, that segment contains the segment $[z, (1+r)z]$. The second and third inclusion in~\eqref{inclusions_Om_eps} imply $[z, (1+r)z]\cap\pa\Om_\veps\neq\emptyset$. This proves the claim.

Lemma~\ref{MongeAmperepointinside} applies to this situation with $s=r$. Since $v(z)-l_\veps(z)=v(o)-l_\veps(o)=-\veps$ (see~\eqref{vminusl}), \eqref{formula_lemmaTW_3} yields
\[
1=\frac{|v(z)-l_\veps(z)|}{|v(o)-l_\veps(o)|}\leq \frac{c_1}{c_0}r^{1/n}.
\] 
Since $r$ can be any positive number, we have reached a contradiction.
%Let $c_0$ and $c_1$ be the constants appearing in Lemma~\ref{MongeAmperepointinside}. 
%It follows from \eqref{uminusl}, (\ref{bttau}) and Lemma~\ref{MongeAmperepointinside}~\eqref{MApi_ii} that
%if $\varepsilon\in(0,\varepsilon_1)$, then
%\begin{equation}
%\label{epsilonb}
%\varepsilon=|v(o)-l_\varepsilon(o)|\geq c_1\mathcal{H}^n(\Omega_\ve)^{1/n}\nu(\Omega_\ve)^{1/n}.
%\end{equation}
%On the other hand, let $s=c_1^n/(2c_0)^n$. It follows from \eqref{SOmegaep},  $\dim\, S\geq 1$ and from the fact that the origin is the 
%centroid of $S$ that
%there exists $\varepsilon\in(0,\ve_1)$ small enough, such that
%$S\subset L\cap \Omega_\varepsilon\subset(1+s)S$. In particular, there exists
%$z_\varepsilon\in S$ such that $(z_\varepsilon+sE_\varepsilon)\cap \partial \Omega_\varepsilon\neq\emptyset$.
%It follows from Lemma~\ref{MongeAmperepointinside}~\eqref{MApi_i} that
%\[
%\varepsilon=|v(z_\varepsilon)-l_\varepsilon(z_\varepsilon)|\leq c_0s\,\mathcal{H}^n(\Omega_\ve)^{1/n}\nu(\Omega_\ve)^{1/n}
%= \frac{c_1}2\,\mathcal{H}^n(\Omega_\varepsilon)^{1/n}\nu(\Omega_\varepsilon)^{1/n}.
%\]
%This contradicts (\ref{epsilonb}), and in turn proves Proposition~\ref{ustrictconvex}. 
\end{proof}

We will actually use the following consequence of Proposition~\ref{ustrictconvex}.
 
\begin{corollary}\label{ChouWang_ustrictconvex}
Let $\tau_2>\tau_1>0$, and let $g$ be a function defined on an open convex set $\Omega\subset\R^n$, $n\geq 2$, such that $\tau_2>g(x)>\tau_1$ for $x\in\Omega$.  For $p<1$, let $v$ be a non-negative convex solution of 
\begin{equation*}
v^{1-p}\det D^2 v= g \quad\text{ in $\Omega$}.
\end{equation*}
If $S=\{x\in \Omega:v(x)=0\}$ is non-empty, compact and $\mu_v(S)=0$, 
and $v$ is locally strictly convex on $\Omega\backslash S$, then $S$ is a point.
\end{corollary}
\begin{proof} All we have to check that ${\rm dim}\, S\leq n-1$. It follows from the fact that the left-hand side of the differential equation is zero on $S$, while the right-hand side is positive.
\end{proof}

The following result by L.~Caffarelli (see Theorem~1 and Corollary~1 in \cite{Caf90a}) is the key in handling the regularity and strict convexity of the part of the boundary 
of a convex body $K$ where the support function at some normal vector is positive.

\begin{theorem}[Caffarelli]
\label{Caffarelli-smooth}
Let $\lambda_2>\lambda_1>0$,  and let $v$ be a convex function on an open convex set $\Omega\subset \R^n$ such that 
\begin{equation}\label{aggiunta}
\lambda_1\leq \det D^2v\leq \lambda_2
\end{equation}
in the sense of measure.
\begin{enumerate}[(i)]
\item If $v$ is non-negative and
$S=\{x\in\Omega:\,v(x)=0\}$ is not a point, then $S$ has no extremal point in $\Omega$.
\item If $v$ is strictly convex, then $v$ is $C^1$.
\end{enumerate}
\end{theorem}
% \noindent{\bf Remark } The condition on $v$ is equivalent saying that 
% $$
% \lambda_1\mathcal{H}^n(U)\leq \mathcal{H}^n(N_v(U))\leq\lambda_2\mathcal{H}^n(U)
% \mbox{ \ for any bounded open $U\subset \Omega$}.
% $$
We recall that \eqref{aggiunta} is equivalent to saying  that for each Borel set $\omega\subset \Omega$ we have
\[
 \lambda_1\mathcal{H}^n(\omega)\leq \mu_v(\omega)\leq \lambda_2\mathcal{H}^n(\omega),
\]
where $\mu_v$ has been defined in~\eqref{Monge-Ampere-measure}.

\section{Proof of Theorem~\ref{th_regularity}}
\label{sec-th_regularity}

The next lemma provides a tool for the proof of Theorem~\ref{th_regularity} (iii). The same result is also proved in Chou and Wang 
\cite{CW}; we present a short argument for the sake of completeness.

\begin{lemma}\label{oinside}
For $n\geq 2$ and $p\leq 2-n$, if $K\in{\mathcal K}_{0}^n$ and there exists $c>0$ such that
  $S_{K,p}(\omega)\geq c\,{\mathcal H}^{n-1}(\omega)$ for any Borel set $\omega\subset S^{n-1}$, then
$o\in{\rm int}\,K$.
\end{lemma}

\begin{proof}  
We suppose that  $o\in\partial K$ and seek a contradiction.  We choose $e\in N(K,o)\cap S^{n-1}$ such that
$\{ \lambda e :\lambda\geq 0\}$ is an extremal ray of $N(K,o)$. Let $H^+$ be a closed half space containing $\R e$ on the boundary such that
$N(K,o)\cap {\rm int} H^+=\emptyset$. Let
$$
V_0=S^{n-1}\cap (e+B^n)\cap {\rm int} H^+.
$$
It follows by the condition on $S_{K,p}$ that
\begin{equation}
\label{honV0}
c\int_{V_0}h_K(u)^{p-1}\,d{\mathcal H}^{n-1}\leq
\int_{V_0}h_K(u)^{p-1}\,dS_{K,p}= S_K(V_0)<\infty.
\end{equation}

However, since $h_K$ is convex and $h_K(e)=0$, there exists $c_0>0$ such that
$$
h_K(x)\leq c_0\|x-e\|\mbox{ \ for $x\in e+B^n$}. 
$$
We observe that the radial projection of $V_0$ onto the tangent hyperplane $e+e^\bot$ to $S^{n-1}$ at $e$
 is $e+V'_0$ for
$$
V'_0=e^\bot\cap (\sqrt{3}\,B^n)\cap {\rm int} H^+.
$$
If $y\in V'_0$, then $u=(e+y)/\|e+y\|$ verifies $\|u-e\|\geq  \|y\|/2$.
It follows that
\begin{eqnarray*}
\int_{V_0}h_K(u)^{p-1}\,d{\mathcal H}^{n-1}&\geq&
c_0^{p-1}\int_{V_0}\|u-e\|^{p-1}\,d{\mathcal H}^{n-1}(u)\\
&\geq&
\frac{c_0^{p-1}}2\int_{V'_0}\frac{\|y\|^{p-1}}{(1+\|y\|^2)^{n/2}}\,d{\mathcal H}^{n-1}(y)\\
&\geq &
\frac{c_0^{p-1}}{2^{n+1}}\int_{V'_0}\|y\|^{p-1}\,d{\mathcal H}^{n-1}(y)=\infty
\end{eqnarray*}
as $p\leq 2-n$. This contradicts \eqref{honV0}, and hence verifies the lemma.
\end{proof}

\begin{proof}[Proof of Theorem~\ref{th_regularity}]
\emph{Claim~\eqref{th_regularity_a_}.} 
For $u_0\in S^{n-1}\backslash N(K,o)$, we choose a spherically convex open neighbourhood $\Omega_0$ of $u_0$ on $S^{n-1}$ such that for any $u\in{\rm cl}\,\Omega_0$, we have $\langle u,u_0\rangle>0$ and $u\not\in N(K,o)$. Let $\Omega\subset u_0^\bot$ be defined in a way such that
$u_0+\Omega$ is the radial image of $\Omega_0$ into $u_0+u_0^\bot$, and let $v$ be the function on $\Omega$  defined as in Lemma~\ref{MongeAmpereRn-lemma} with $h=h_K$. Since $h_K$ is positive and continuous on ${\rm cl}\,\Omega$, we deduce from Lemma~\ref{MongeAmpereRn-lemma} that there exist $\lambda_2>\lambda_1>0$ depending on $K$, $u_0$ and $\Omega_0$ such that
\begin{equation}
\label{vzinXsmooth}
\lambda_1\leq \det D^2v\leq \lambda_2
\end{equation}
on $\Omega$. 

First we claim that
\begin{equation}
\label{zinXsmooth}
\mbox{ if  $z\in\partial K$ and $N(K,z)\not\subset N(K,o)$, then
$z$ is a $C^1$-smooth point.}
\end{equation}
We suppose that ${\rm dim}\,N(K,z)\geq 2$, and seek a contradiction. Since $N(K,z)$ is a closed convex cone such that $o$ is an extremal point, the property $N(K,z)\not\subset N(K,o)$ yields an 
$e\in (N(K,z)\cap S^{n-1})\backslash N(K,0)$ generating an extremal ray of $N(K,z)$. 
We apply the construction above for $u_0=e$. The convexity of $h_K$ and~\eqref{duality_body_support1}  imply 
$h_K(x)\geq \langle z,x\rangle$ for $x\in\R^n$, with equality if and only if $x\in N(K,z)$.
We define $S\subset \Omega$ by $S+e=N(K,z)\cap(\Omega+e)$ and hence $o$ is an extremal point of $S$.
It follows that the function $\tilde{v}$  defined by $\tilde{v}(y)=v(y)-\langle z,y+e\rangle$ is non-negative on $\Omega$,
satisfies \eqref{vzinXsmooth}, and
$$
S=\{y\in\Omega:\, \tilde{v}(y)=0\}.
$$
These properties contradict Caffarelli's Theorem~\ref{Caffarelli-smooth} (i) as  $o$ is an extremal point of $S$, and in turn we conclude \eqref{zinXsmooth}.

Next we show that
\begin{equation}
\label{u0notinNKo}
\mbox{ $h_K$ is differentiable at any $u_0\in S^{n-1}\backslash N(K,o)$.}
\end{equation}
We apply again the construction above for $u_0$. If $u\in\Omega_0$ and $z\in F(K,u)$ clearly $K$ is $C^1$-smooth at $z$ (i.e. $N(K,z)$ is a ray) by~\eqref{zinXsmooth}. Therefore, by~\eqref{notstrictlyconvex},  $v$ is strictly convex on $\Omega$ and Caffarelli's Theorem~\ref{Caffarelli-smooth}~(ii) yields that $v$ is $C^1$ on $\Omega$. In turn, we conclude \eqref{u0notinNKo}. 

In addition, $F(K,u)$ is a unique $C^1$-smooth point for $u\in\Omega_0$ (see \eqref{duality_body_support2}), yielding
that $\Omega_*=\cup\{F(K,u):\,u\in \Omega_0\}$ is an open subset of  $\pa K$. 
Therefore $\Omega_*\subset X$, any point of $\Omega_*$ is $C^1$-smooth 
(by \eqref{notstrictlyconvex}) and $\Omega_*$ contains no segment
(by \eqref{duality_body_support2}),
  completing the proof of Claim~\eqref{th_regularity_a_}.
 
\emph{Claim~\eqref{th_regularity_b_}.} We suppose that $o\in\pa K$ is $C^1$-smooth, and there exists
 $z\in \pa K$ such that $K$ is not $C^1$-smooth at $z$.
 Claim~\eqref{th_regularity_a_} yields that $z\in X_0$, and hence $N(K,z)\subset N(K,o)$, which is a contradiction, verifying Claim~\eqref{th_regularity_b_}.

\emph{Claim~\eqref{th_regularity_c_}.} This is a consequence of Lemma~\ref{oinside} and Claim~\eqref{th_regularity_a_}.

\emph{Claim~\eqref{th_regularity_x_}.} This is a consequence of Lemma~\ref{MongeAmpereRn-lemma}, Claim~\eqref{th_regularity_a_} and Caffarelli \cite{Caf90b}.

\end{proof}

\begin{example}\label{non-strictly-convex} If $n\geq2$ and $p\in(-n+2,1)$, then there exists $K\in{\mathcal K}_{0}^n$ with $C^1$ boundary such that $o$ lies in the relative interior of a facet of $\partial K$ and $dS_{K,p}=f\,d{\mathcal H}^{n-1}$ for a strictly positive continuous $f:\,S^{n-1}\to\R$.

Let $q=(p+n-1)/(p+n-2)$. We have $q>1$. Let 
\[
 g(r)=\begin{cases}
         (r-1)^q&\text{when $r\geq1$;}\\
         0&\text{when $r\in[0,1)$;}
       \end{cases}
\]
and $\bar g(x_1,\dots,x_{n-1})=g(\|(x_1,\dots,x_{n-1})\|)$. Let $K\in{\mathcal K}_{0}^n $ be such that $K\cap \{x : x_n\leq1\}=\{x : 1\geq x_n\geq \bar g(x_1,\dots,x_{n-1})\}$ and $\d K\cap \{x : x_n>0\}$ is a $C^2$ surface with Gauss curvature positive at every point.
Clearly $K\cap \{x : x_n=0\}$ is a $(n-1)$-dimensional face of $K$ which contains $o$ in its relative interior and has unit outer normal $(0,\dots,0,-1)$. 

To prove that $dS_{K,p}=f\,d{\mathcal H}^{n-1}$ for a positive continuous $f:\,S^{n-1}\to\R$, it suffices to prove that there is a neighbourhood of the South pole where $dS_{K,p}/d{\mathcal H}^{n-1}$ is continuous and bounded from above and below by positive constants. 
Let $h$ be the support function of $K$ and, for $y\in\R^{n-1}$, let $v(y)=h(y,-1)$ be the restriction of $h$ to the hyperplane tangent to $S^{n-1}$ at the South pole. It suffices to prove that  in a neighbourhood $U$ of $o$, $v$ satisfies the equation $v^{1-p}\det D^2 v=G$ with a function $G$ which is bounded from above and below by positive constants.

If $y\in U\setminus\{o\}$ we have 
\begin{equation}\label{non-strictly_convex_f1}
 v(y)=h(y,-1)=\langle (x',\bar g(x')),(y,-1)\rangle\quad\text{ where }\quad D \bar g(x')=y.
\end{equation}
If $U$ is sufficiently small then $v(y)$ depends only on $\|y\|$. Let $y=(z,0,\dots,0)$, with $z>0$ small and let $r=1+(z/q)^{1/(q-1)}$. We have 
\[
 D \bar g(r,0,\dots,0))=(z,0,\dots,0)
\]
and \eqref{non-strictly_convex_f1} gives
\begin{align*}\label{non-strictly_convex_f2}
 v(z,0,\dots,0)=&r q(r-1)^{q-1}-(r-1)^q\\
 =&z+\frac{q-1}{q^{n-1+p}}z^{n-1+p}.
\end{align*}
(Note that $n-1+p>1$.) Clearly $v(0,\dots,0)=h(0,\dots,0,-1)=0$.
When $z>0$ we have 
\begin{align*}
 &v_{y_1y_1}=\frac{q-1}{q^{n-1+p}}(n-1-p)(n-2-p)z^{n-3+p}\\
 &v_{y_iy_i}=\frac1{z}+\frac{q-1}{q^{n-1+p}}(n-1-p)z^{n-3+p} \quad\quad\quad\text{when $i\neq1$}\\
 &v_{y_iy_j}=0\quad\quad\quad\text{when $i\neq j$,}
\end{align*}
and, as $z\to0^+$
\[
v(z,0,\dots,0)^{1-p}\det D^2 v(z,0,\dots,0)=c+o(1),
\]
for a suitable constant $c>0$.
This implies the existence of a function $G$ positive and continuous on $U$ such that
\[
\mathcal{H}^{n-1}\big(N_v(\omega\cap\{v>0\})\big)=\int_{\omega\cap\{v>0\}} nG(y)v(y)^{p-1}\ dy.
\]
for any Borel set $\omega\subset U$.  To conclude the proof that $v$ is a solution in the sense of Alexandrov of $v^{1-p} \det D^2 v=G$ in $U$ it remains to prove that $\mathcal{H}^{n-1}\big(\{y\in U : v(y)=0\}\big)=0$, but this is obvious since $\{y\in U : v(y)=0\}=\{o\}$.

We remark that $h$ is not a solution of~\eqref{MongeAmpere_CW} because~\eqref{sol_alex_c} fails.
\end{example}

\section{Proofs of Theorem~\ref{th_regularity-new} and Corollary~~\ref{cor_regularity_new}}
\label{sec-th_regularity-new}

\begin{proof}[Proof of Theorem~\ref{th_regularity-new}]
We may assume that $o\in \pa K$ since otherwise $\partial K$ is $C^1$  by Theorem~\ref{th_regularity}. Let $e\in N(K,o)\cap S^{n-1}$ be such that $\langle u,e\rangle>0$ for any
$u\in N(K,o)\cap S^{n-1}$. Let  $v$ be defined on $\Omega=e^\bot$ as in Lemma~\ref{MongeAmpereRn-lemma} with $h=h_K$ and let $S=\{x\in e^\bot:\,v(x)=0\}$. We have
\begin{equation}\label{th_reg_new_a}
 S+e=N(K,o)\cap(e^\bot +e),
\end{equation}
by~\eqref{duality_body_support1}.
If $K$ is not $C^1$-smooth at $o$ then $\dim S\geq 1$ and, by Proposition~\ref{vanishingonsegment}, $p\geq n-4$ (note that here the dimension of the ambient space is $n-1$). This proves Theorem~\ref{th_regularity-new}~(i).

To prove Theorem~\ref{th_regularity-new}~(ii) we observe that
\[
 N_{h_K}(e+S)=\bigcup_{u\in N(K,o)}F(K,u)=X_0,
\]
where $X_0$ is defined as in Theorem~\ref{th_regularity}~(i). The equality on the left in this formula follows by~\eqref{duality_body_support2} and the equality on the right follows by Theorem~\ref{th_regularity}~(i).
Thus,
$$
N_v(S)=X_0|e^\bot,
$$
and if $\mathcal{H}^{n-1}(X_0)=0$ then $\mu_v(S)=0$. We observe that $S$ is compact, by~\eqref{th_reg_new_a}, that $v$ is locally strictly convex, by
Theorem~\ref{th_regularity}~(i), and that ${\rm dim}\,S\leq n-2$, by
\eqref{NKosmall}.
Hence,  Theorem~\ref{th_regularity-new}  (ii) follows by Corollary~\ref{ChouWang_ustrictconvex} and ~\eqref{th_reg_new_a}. 
\end{proof}

\begin{proof}[Proof of Corollary~\ref{cor_regularity_new}]
Claim~\eqref{th_regularity_g_} is an immediate consequence of~\eqref{duality_body_support1}, Proposition~\ref{vanishingonsegment} and Lemma~\ref{MongeAmpereRn-lemma}. This claim implies that when $n=4$ or $n=5$ and $\partial K$ is not $C^1$ then $\dim N(K,o)=2$. In this case  $N(K,o)\cap S^{n-1}$ is a closed arc: let $e_1$ and $e_2$ be its endpoints. If $u\in N(K,o)\cap S^{n-1}$, $u\neq e_1$, $u\neq e_2$, then $F(K,u)$ is contained in the intersection of the two supporting hyperplanes $\{x\in\R^n: \langle x,e_i\rangle=h_K(e_i)\}$, $i=1,2$. Thus,
\[
 \mathcal{H}^{n-1}\Big(\bigcup\{F(K,u) : u\in N(K,o)\cap S^{n-1}, u\neq e_1, u\neq e_2\}\Big)=0.
\]
Therefore  $\dim F(K,e_1)=n-1$ or $\dim F(K,e_2)=n-1$, because otherwise 
\[
\bigcup\{F(K,u) : u\in N(K,o)\cap S^{n-1}\},
\]
which coincides with $X_0$ by~Theorem~\ref{th_regularity}~(i), has $(n-1)$-dimensional Hausdorff measure equal to zero and $\partial K$ is $C^1$ by Theorem~\ref{th_regularity-new}~(ii).
\end{proof}

\noindent{\bf Acknowledgement } We are grateful to the referees. Their observations substantially improved the paper.

\end{document}